\documentclass{amsart}
\usepackage{ amssymb }
\usepackage{geometry} 

\def\R{\mathbb{R}}

\def\N{\mathbb{N}}
\def\S{\mathbb{S}}

\newtheorem{theo}{Theorem}[section]

  \newtheorem{cor}{Corollary}[section]
    \newtheorem{prop}{Proposition}[section]
        \newtheorem{lem}{Lemma}[section]

\usepackage{hyperref}

\begin{document}

\title{ Yamabe Flow with prescribed scalar curvature }

\author{Inas Amacha}
\address{Laboratoire de Math\'ematiques, UMR 6205 CNRS
Universit\'e de Bretagne Occidentale
6 Avenue Le Gorgeu, 29238 Brest Cedex 3
France}
\email{Inas.Amacha@univ-brest.fr}
\author{Rachid Regbaoui }
\address{Laboratoire de Math\'ematiques, UMR 6205 CNRS
Universit\'e de Bretagne Occidentale
6 Avenue Le Gorgeu, 29238 Brest Cedex 3
France}
\email{Rachid.Regbaoui@univ-brest.fr}


\begin{abstract}
In this work, we study the Yamabe flow corresponding to the prescribed scalar curvature problem on compact Riemannian manifolds with negative scalar curvature. The long time  existence and convergence of the flow are proved  under appropriate conditions on the prescribed scalar curvature function.
 \end{abstract}
\maketitle
\section{Introduction}

\bigskip

The prescribed scalar curvature problem on a compact  Riemannian manifold $(M, g_0)$ of dimension $n \ge 3$, consists of finding a conformal metric $g$ to $g_0$   whose scalar curvature $R_g$   is equal to a given function $f \in C^{\infty}(M)$.  If we set $g = u^{4\over n-2}g_0$, where $ 0< u \in C^{\infty}(M)$, then we have
$$ R_g = u^{-{n+2\over n-2}} \left(-c_n\Delta u + R_0u \right)$$
where $\Delta$ is the Laplace operator associated with $g_0$, $R_0$ is the scalar curvature of $g_0$ and $c_n = 4 {n-1\over n-2}$.

\bigskip

Then the prescribed scalar curvature problem :
$$R_g = f $$

\noindent  is equivalent to solve the following nonlinear PDE :

$$-c_n\Delta u +  R_0 u  = fu^{n+2 \over n-2} \eqno (1.1)$$

\medskip

\noindent on the space of smooth positive functions on $M$.  The solvability of this equation depends on $R_0$ and the prescribed function $f$. When $f$ is constant, equation (1.1) becomes the famous Yamabe equation whose resolution has been for a  long time a challenging problem in geometric Analysis. See \cite{tA}, \cite{eH},  \cite{jL}, \cite{rS1}, \cite{rS2},  for more details on the Yamabe problem, and \cite{aA}, \cite{sB}, \cite{ jE}, \cite{jK}, \cite{aR}, \cite{jV},  concerning the prescribed scalar curvature problem.

\medskip

By changing  conformally  $g_0$ if necessary, we may always assume that $R_0$ satisfies one of the  conditions :  $R_0 > 0\ , \  R_0 = 0$ \  or \ $R_0 < 0$  everywhere on $M$. Equation (1.1) has a variational structure since  there are different functionals whose Euler-Lagrange equations are equivalent to (1.1). When $R_0 < 0$,  the following functional  seems more appropriate to  handle the prescribed scalar curvature  problem :

$$ {\mathcal E}(g) = \int_M R_g dV_g - \frac{n-2}{ n} \int_M fdV_g  \eqno(1.2) $$

\noindent  where $g=  u^{4\over n-2}g_0$ belongs to the conformal class  $ [g_0]$  of $g_0$,  $R_g$ is the scalar curvature of $g$  and $ dV_g = u^{2n\over n-2}dV_{g_0}$ is the volume element of $g$.

 \medskip

Simple computations (\cite{aB}) show that the $L^2$-gradient of ${\mathcal E}$ is  \  $ {n-2 \over 2n}(R_g-f)g$,
 and then, after changing time by a constant scale,  the associated negative gradient flow equation is

  \medskip

  $$\begin{cases} \partial_t g = - (R_g-f)g  \cr
  g(0) = g^0
  \end{cases}  \eqno (1.3)$$
  where $g^0 = u_0^{4\over n-2}g_0$ is a given metric in the conformal class of $g_0$.

  \medskip

   Since equation (1.3) preserves the conformal structure of $M$, then any  smooth solution of (1.3) is of the form  $g(t) = u(t)^{4\over n-2} g_0$, where $0< u(t) \in C^{\infty}(M)$. For simplicity we have used the notation $u(t):= u(t, .), \ t \in I$  for any function $u$ defined on $I \times M$, where $I$ is a subset of $\R$.   In terms of $u(t)$,  the flow  (1.3) may be written in the equivalent form :

  \medskip

  $$ \begin{cases} \partial_tu^{N} = {n+2 \over 4}\left(c_n\Delta u - R_0u + fu^{N}\right)   \cr
   u(0) = u_0 \in C^{\infty}(M), \ u_0 >0,
   \end{cases} \eqno (1.4) $$
   where $N= {n+2 \over n-2}$.

  \medskip

  Our aim in this paper  is to investigate this gradient flow by proving its longtime existence and analysing its asymptotic   behavior when $t \to + \infty$.
  \bigskip

  Our first result is the  following existence theorem  :
  \medskip

  \begin{theo} Suppose that $R_0 <0$ and  let  $f \in C^{\infty}(M)$. Then  for any  $g^0 = u_0^{4\over n-2} g_0$ with $ 0< u_0 \in C^{\infty}(M)$, there exists a unique solution $g(t) = u(t)^{4\over n-2}g_0$ of (1.3) defined on $[0, +\infty)$, where $0< u \in C^{\infty}( [0, +\infty)\times M)$. Moreover, the functional ${\mathcal E}$ is decreasing along the solution $g(t)$, that's
  $${d\over dt}{\mathcal E}(g(t)) \le 0  \  \    \hbox{for all} \  t \in [0, +\infty) .$$
  \end{theo}

\bigskip

We note here that apart from  the smoothness of $f$, no further assumptions on the function $f$ are needed in Theorem 1.1. However, for the longtime  behavior,  it is necessary to assume additional condition in order to  get the convergence of the flow. Indeed, if $ f\ge 0$,  by applying the maximum principe to  (1.4), one  can easily check that

$$ u(t) \ge \left(\min_{M}|R_0|\min_M u_0\right)t^{n-2\over 4} \to + \infty \ \  \hbox{as} \  t \to + \infty .$$
So if one wants to get the convergence of the flow, it is necessary to assume at least that $f$ is negative somewhere on $M$. We note that this last condition is also  necessary to the resolution  of equation (1.1)  since it is well known that it  if the negative gradient flow associated with a functional ${\mathcal F}$ converges (in some sense), then its limit is a critical point of ${\mathcal F}$.

\medskip

Before giving conditions on  $f$ ensuring the convergence of the flow, let us fix some notations:  if $\Omega \subset M$ is an open set, we denote by $\lambda_{\Omega}$ the first eigenvalue of the conformal Laplacian $L = -c_n\Delta + R_0 $   on $\Omega$ with zero Dirichlet boundary conditions, that is
$$\lambda_{\Omega} = \inf_{0\not=u\in H_0^1(\Omega)}{\int_M (c_n\vert\nabla u\vert^2 + R_0 u^2) dV_{g_0} \over \int_M u^2 dV_{g_0 }}. $$

\medskip

\noindent We then assume the following  conditions on $f$ :

\medskip

\medskip

There exists an open set $\Omega \subset M$ such that

$$     \lambda_{\Omega} >  0   \  \    \hbox{and} \  \   f < 0 \  \hbox{on} \  M\setminus \Omega   \leqno  \bf{(H1) } $$

 \noindent and

$$   \sup_{x\in \Omega} f(x) \le C_{\Omega}\inf_{x\in M\setminus\Omega}\left|f(x)\right|,  \leqno \bf{(H2)}$$

\medskip

\noindent where $C_{\Omega}$ is a positive constant depending only on $\Omega$.

\bigskip

We then have the following result :

\medskip

\begin{theo} Suppose that $R_0 <0$ and that $f \in C^{\infty}(M)$ satisfies conditions $(H1)-(H2)$. Then there exists a function $ 0< {\bar u} \in C^{\infty}(M)$ such that for  any  smooth metric  $g^0 = u_0^{4\over n-2} g_0$ with  $0<u_0 \le {\bar u}$, the flow $g(t) = u(t)^{4\over n-2}g_0$ given by Theorem 1.1 converges in the \ $C^{\infty}$-topology to a conformal metric $g_{\infty} = u_{\infty}^{4\over n-2}g_0$ whose scalar curvature is $f$, that's, \ $R_{g_{\infty}} = f$.
\end{theo}

\bigskip

A particular interesting case is when the function $f$ satisfies $f(x) <  0$ for almost all $x \in M$. In this  case conditions (H1)-(H2) are automatically satisfied and then we have  the following corollary:

\medskip

\begin{cor} Suppose that $R_0 <0$ and  $f\in C^{\infty}(M)$ such that  $f < 0$ almost everywhere on $M$. Then there exists a function $ 0< {\bar u} \in C^{\infty}(M)$ such that for  any  smooth metric  $g^0 = u_0^{4\over n-2} g_0$ with  $0<u_0 \le {\bar u}$, the flow $g(t) = u(t)^{4\over n-2}g_0$ given by Theorem 1.1 converges in the \ $C^{\infty}$-topology to a conformal metric $g_{\infty} = u_{\infty}^{4\over n-2}g_0$ whose scalar curvature is $f$, that's, \ $R_{g_{\infty}} = f$.
\end{cor}

\bigskip

It is naturel to ask if conditions (H1)-(H2) in Theorem 1.2 are necessary. The following Theorem tells us that this is the case at least for condition (H1).

\bigskip

\begin{theo} Suppose that $R_0<0$ and let  $f \in C^{\infty}(M)$ such that condition (H1) is not satisfied, that's, for any open set \ $\Omega \subset M$ such that  $f > 0$ on  $M \setminus \Omega$,  we suppose  $\lambda_{\Omega} \le 0$. Then for any \  $0< u_0 \in C^{\infty}(M)$, the solution $u(t)$ of (1.4) satisfies  for some constant $C>0$ depending only on $u_0, g_0, f$,
$$\max_{x\in M}u(t,x)  \ge Ct^{n-2\over n+2}  \to +\infty \  \ \hbox{as} \  t \to +\infty .$$
\end{theo}

\bigskip

\bigskip

We note here that condition (H1)  is conformally invariant. Similar conditions to (H1)-(H2) were found by many authors to solve  (1.1) by the direct method of elliptic PDEs, see (\cite{sB}, \cite{aR}, \cite{jV}) for more details.  To our knowledge, the only known results on  Yamabe type flow on dimension $n \ge 3$ concern the case where $f$ is constant  or  $M= \S^n$. The Yamabe flow  was first introduced by Hamilton \cite{rH}  and has been the subject  of several studies, see  \cite{sB1}, \cite{sB2},  \cite{bC}, \cite{hS}, \cite{rY}. When   $f$ is non constant,  we mention the work of Struwe \cite{mS} about  the Nirenberg's problem on the sphere $\S^2$, and   the results of Chen-Xu \cite{xC}  concerning $\S^n$, $n \ge 3$. A general  evolution problem related to the prescribed Gauss curvature on surfaces   was studied by Baird-Fardoun-Regbaoui \cite{pB}.

\bigskip

\medskip

The paper is organized as follows. In section 2 we prove the global existence of the flow by establishing local $C^k$-estimates on the solution $u$ of (1.4). In section 3, we  study the asymptotic behavior of the flow when $t \to+\infty$. In particular we prove  uniform $C^k$-estimates on $u$ which are necessary to  get the convergence of the flow.

\bigskip

\section{global existence of the flow}

\bigskip

In this section we shall establish some estimates on the solution $u$ of (1.4) which will be an important tool in proving that the flow $g(t)$ is globally defined on $[0,+\infty)$. In this section we suppose that $R_0 <0$  and $f \in C^{\infty}(M)$.

\bigskip

As already mentioned in the previous section,  equation (1.3)  is equivalent to (1.4), so it suffices to prove  the existence of a solution $u(t)$ of (1.4) defined on $[0, +\infty)$ to obtain a metric $g(t)$ solution of (1.3)  defined on $[0, +\infty)$.
Since (1.4) is a parabolic equation (on the set of smooth positive functions on $[0, T)\times M$, for any $T>0$), then there exists a smooth solution $u(t)$ of
(1.4) defined on a maximal interval $[0, T^*)$ satisfying  $u(t) > 0$ on $[0, T^*)$.  Thus we have a solution $g(t) = u^{4\over n-2} g_0$ of (1.3) defined on a maximal interval $[0, T^*)$. For simplicity, we shall write $u$ instead of $u(t)$ and $g$ instead of $g(t)$.

\medskip

Now, we derive some properties on $g$ which will be important later. One can check by  using  (1.4)  that the scalar curvature $R_g$ satisfies the following equation

$$\partial_t R_g = (n-1) \Delta_g(R_g -f) + R_g(R_g-f) \eqno (2.1) $$

\medskip

\noindent where $\Delta_g$ is the Laplacian associated with $g(t)$.

\medskip

\medskip

A simple computation using (2.1) gives

$${d \over dt}{\mathcal E}(g) = -{n-2\over 2}\int_{M}(R_g-f)^2 dV_g,  \eqno (2.2)$$

\noindent so the functional ${\mathcal E}$ is decreasing along the flow $g(t)$. If we set
$$E(u) := {\mathcal E}(g) = {\mathcal E}\left(u^{4\over n-2}g_0\right) = \int_{M} \left(c_n\vert \nabla u\vert^2 + R_0u^2  -{n-2\over n}  f u^{2n\over n-2}\right) dV_{g_0},  $$

\noindent then (2.2) can be written in terms of $u$ :

$${d \over dt}E(u) = -{8\over n-2} \int_{M}\vert\partial_tu\vert^2 u^{4\over n-2}dV_{g_0} \le 0.  \eqno (2.3)$$

\medskip

\medskip

The following lemma will be very useful to prove integral estimates on the solution $g$.

\begin{lem} We have for any $p >1$,
 $${d \over dt}\int_M |R_g-f|^p dV_g =  - {4(n-1)(p-1)\over p}\int_M \left\vert\nabla_{\hspace{-0.7mm}g} |R_g-f|^{p\over2} \right\vert_g^2dV_g $$
 $$+ \    \left(p-{n\over 2}\right)\int_M (R_g-f)|R_g-f|^p dV_g
+ \ p\int_M f|R_g-f|^p dV_g ,$$

  \noindent  where $\nabla_{\hspace{-0.7mm}g}$ is the gradient with respect to the metric $g$ and $| \ . \  |_{g}$ is the Riemannian norm with respect to $g$.

 \end{lem}

 \medskip

 \begin{proof} We have for any $p\ge 1$

 $${d \over dt}\int_M |R_g-f|^p dV_g  =p\int_M  |R_g-f|^{p-2}(R_g-f)\partial_tR_g dV_g  + {1\over 2}\int_M  R_g  \hspace{0.5mm} tr_g(\partial_t g) dV_g .$$

 \medskip

 \noindent Using equations (1.3) and (2.1) it follows that

  $${d \over dt}\int_M |R_g-f|^p dV_g = (n-1)p\int_M |R_g-f|^{p-2}(R_g-f)\Delta_g(R_g-f)  $$
 $$ + \ p \int_M R_g|R_g-f|^{p} -{n\over 2}\int_M  |R_g-f|^p(R_g-f)dV_g  $$
 $$= - {4(n-1)(p-1)\over p}\int_M \left\vert\nabla_{\hspace{-1mm}g} |R_g-f|^{p\over2} \right\vert_g^2dV_g +  \left(p-{n\over 2}\right)\int_M (R_g-f)|R_g-f|^p dV_g$$
 $$+ \ p\int_M f|R_g-f|^p dV_g .$$

 \end{proof}

\medskip

\medskip

In order to  prove that the solution $g(t) = u(t)^{4\over n-2}g_0$ is globally defined on $[0, + \infty)$, we need upper and lower bounds on $u(t)$. \medskip

\begin{prop} Let $g(t) =  u(t)^{4\over n-2}g_0$ be the solution of $(1.3)$ defined on a maximal interval $[0, T^*)$. Then we have for any $t\in [0, T^*)$ :

$$  \min\left(C_0 , \min_{M} u_0 \right)   \le  u(t) \le  \max\left(1, \max_M u_0\right)e^{C_1t}, \eqno (2.4)$$

\noindent where $  \displaystyle C_0 = \left({\min_{M} |R_0| \over \max_{M}|f| }\right)^{n-2\over 4}\hspace{-1.5mm},  \   C_1 = {n-2\over 4}\left( \max_M|R_0| +  \max_M |f|\right)$ .
\end{prop}

\begin{proof}
The proof uses an elementary maximum principal argument. Indeed, fix $t \in [0, T)$ and let $(t_0, x_0) \in [0, t]\times M$ such that $ \displaystyle u(t_0, x_0) = \min_{ [0,t]\times M}u $. If $t_0 = 0$, then
$$ \min_{ [0,t]\times M}u= \min_{ M}u_0 , $$

\noindent so the first inequality in (2.4) is proved in this case. Now suppose that $t_0 >0$. We have then $\partial_t u(t_0, x_0) \le 0$ and  $\Delta u(t_0, x_0) \ge 0$. Thus we obtain after substituting in (1.4) that

$$0 \ge -R_0(x_0)u(t_0, x_0) + f(x_0)u^{N}(t_0,x_0) $$
which implies that

$$ u(t_0, x_0) \ge \left({\min_{M} |R_0| \over \max_{M}|f| }\right)^{1\over N-1}, $$
where $N = {n+2 \over n-2}$. \ This proves the first inequality in (2.4).  In order to prove the second inequality we set    $v= e^{-C_1t}u$ instead of $u$, where $\displaystyle C_1= {4\over n-2}\left(\max_M|R_0| + \max_M |f|\right) $. As above, fix $t \in [0, T)$ and let $(t_0, x_0) \in [0, t]\times M$ such that $ \displaystyle v(t_0, x_0) = \max_{ [0,t]\times M}v $. If $t_0 = 0$, then  \ $ \displaystyle \max_{ [0,t]\times M}v= \max_{ M}u_0$, which implies
$$ \max_{ [0,t]\times M}u  \le  \max_{ M}u_0 \  e^{C_1t},$$

\noindent so the second inequality in (2.4) is proved in this case. Now suppose that $t_0 >0$. We have then $\partial_t v(t_0, x_0) \ge 0$ and  $\Delta v(t_0, x_0) \le 0$, that's, $ \partial_tu(t_0,x_0) \ge C_1 u(t_0,x_0)$ and    $\Delta u(t_0, x_0) \le 0 $.  We obtain after substituting in (1.4) that

$$NC_1 u^N(t_0, x_0)  \le {n+2 \over 4}\left(-R_0(x_0)u(t_0, x_0) + f(x_0)u^{N}(t_0,x_0) \right)$$
which implies that
$$u(t_0, x_0) \le 1 \eqno (2.5)$$

\medskip

\noindent since $\displaystyle NC_1 =  {n+2 \over 4}  \left(\max_M|R_0| + \max_M |f| \right)$.  It is clear that (2.5) implies that
$$\max_{[0,t]\times M }u \le e^{C_1t} .$$
 The Proof of Proposition 2.1 is then complete.

\end{proof}

\bigskip

Now we prove integral estimates on $R_g$ which will imply estimates on  $\partial_t u $ :

\medskip

\begin{prop}
Let $g(t)$ be the solution of $(1.3)$ defined on a maximal interval $[0, T^*)$. Then  we have for any   $t \in [0, T^*)$,
 $$\int_M |R_{g(t)}-f|^{p} dV_{g(t)} \le Ce^{Ct} \eqno (2.6) $$
 where  $p = {n^2\over 2(n-2)}$  and $C$ is a positive constant depending only on $f, g_0 , u_0$.
 \end{prop}

 \medskip

 \begin{proof} In what follows \  $C$ denotes a positive constant depending on $f, g_0 , u_0$,  whose value may change from line to line.  



 \medskip



\medskip

We have by Lemma 2.1  for any $t\in [0,  T^*)$

 $${d \over dt}\int_M |R_g-f|^p dV_g =  - {4(n-1)(p-1)\over p}\int_M \left\vert\nabla_{\hspace{-0.7mm}g} |R_g-f|^{p\over2} \right\vert_g^2dV_g $$
 $$+  \  \left(p-{n\over 2}\right)\int_M (R_g-f)|R_g-f|^p dV_g
 + \ p\int_M f|R_g-f|^p dV_g , \eqno (2.7)$$

 \noindent  where $\nabla_{\hspace{-0.7mm}g}$ is the gradient with respect to the metric $g$ and $| \ . \  |_{g}$ is the Riemannian norm with respect to $g$.

 \medskip

 $${d \over dt}\int_M |R_g-f|^p dV_g  + {4(n-1)(p-1)\over p}\int_M \left\vert\nabla_{\hspace{-1mm}g} |R_g-f|^{p\over2} \right\vert_g^2dV_g$$
  $$\le \  \left|p-{n\over 2}\right|\int_M |R_g-f|^{p+1} dV_g  + C \int_M |R_g-f|^p dV_g .  \eqno (2.8) $$

 By (2.4) we have

 $$ \int_M \left\vert\nabla_{\hspace{-1mm}g} |R_g-f|^{p\over2} \right\vert_g^2dV_g =  \int_M  \left\vert\nabla |R_g-f|^{p\over2} \right\vert^2  u^2 dV_{g_0}
  \ge C\int_M \left\vert\nabla |R_g-f|^{p\over2} \right\vert^2 dV_{g_0} \eqno (2.9) $$
  and
  $$  \int_{M} |R_g - f|^{p} dV_{g} =  \int_{M} |R_g - f|^{p} u^{2n\over n-2}dV_{g_0}  \ge C  \int_{M} |R_g - f|^{p} dV_{g_0} . \eqno (2.10) $$

  \medskip

  \medskip

  By Sobolev's inequality we have
  $$ \left(\int_{M} |R_g - f|^{pn\over n-2} dV_{g_0}\right)^{n-2 \over n} \le  C \left(\int_M \left\vert\nabla |R_g-f|^{p\over2} \right\vert^2 dV_{g_0} +  \int_{M} |R_g - f|^{p} dV_{g_0}\right)$$

 \noindent  which gives by using  (2.4) that

   $$ \left(\int_{M} |R_g - f|^{pn\over n-2} dV_{g}\right)^{n-2 \over n} \le  Ce^{Ct} \left(\int_M \left\vert\nabla |R_g-f|^{p\over2} \right\vert^2 dV_{g_0} +  \int_{M} |R_g - f|^{p}
 dV_{g_0}\right).  \eqno (2.11) $$

 \medskip

 It follows from (2.8),  (2.9) ,  (2.10)  and (2.11) that

 $$ {d \over dt}\int_M |R_g-f|^p dV_g  + C^{-1} e^{-Ct} \left(\int_{M} |R_g - f|^{pn\over n-2} dV_{g}\right)^{n-2 \over n} $$
 $$   \le  \  \left(p-{n\over 2}\right)\int_M |R_g-f|^{p+1} dV_g  + C \int_M |R_g-f|^p dV_g .  \eqno (2.12) $$

 \medskip

 \bigskip

 By taking  $p= {n\over 2}$ in (2.12) we get

 $$ {d \over dt}\int_M |R_g-f|^p dV_g  \le  C \int_M |R_g-f|^p dV_g $$

 \noindent which implies that

 $$\int_{M} |R_g-f|^{n\over 2} dV_g \le Ce^{Ct} . \eqno (2.13)$$

 \medskip

 Now taking again $p= {n\over 2}$ in (2.12) and  integrating  on $[0, t], t \in [0, T^*)$,   we obtain  by using (2.13)

 $$ \int_0^t\left(\int_{M} |R_{g(s)} - f|^{n^2\over 2(n-2)} dV_{g(s)}\right)^{n-2 \over n} \hspace{-1,5mm}ds  \  \le \   Ce^{Ct} . \eqno (2.14) $$

\medskip

\medskip

We have by H\"older's  inequality and Young's inequality,  for any $\varepsilon >0$   and $p > {n\over 2}$,

$$ \int_M |R_g-f|^{p+1} dV_g  \le \varepsilon \left(\int_{M} |R_g - f|^{pn\over n-2} dV_{g}\right)^{n-2 \over n} +
\varepsilon^{-{n\over 2p-n}} \left( \int_M |R_g - f|^{p} dV_{g}\right)^{2p-n+ 2 \over 2p-n} . \eqno (2.15)$$

\medskip

\bigskip

If we combine  (2.15) with (2.12) and  taking $\varepsilon = (p-{n\over 2})^{-1}C^{-1}e^{-Ct}$,   we get

$${d \over dt}\int_M |R_g-f|^p dV_g  \le  Ce^{Ct}  \left( \int_M |R_g - f|^{p} dV_{g}\right)^{2p-n+ 2 \over 2p-n} +   C \int_M |R_g - f|^{p} dV_{g} $$
that's
$${d \over dt}\log\left(\int_M |R_g-f|^p dV_g\right) \le C\left(e^{Ct}  \left( \int_M |R_g - f|^{p} dV_{g}\right)^{ 2 \over 2p-n} + 1\right) $$

\medskip

In particular by choosing $p = {n^2 \over 2(n-2)}$ and integrating on  $[0, t], \ t \in [0, T^*)$,  we obtain

$$\log\left(\int_M |R_{g(t)}-f|^{n^2\over 2(n-2)} dV_{g(t)}\right) \le    \log\left(\int_M |R_{g(0)}-f|^{n^2\over 2(n-2)} dV_{g(0)}\right) $$
$$ + \ Ce^{Ct}\int_0^t \left( \int_M |R_{g(s)} - f|^{n^2\over 2(n-2)} dV_{g(s)}\right)^{ n-2 \over n}\hspace{-1,5mm}ds  \ + \  Ct $$

\noindent which by using (2.14) gives

$$\log\left(\int_M |R_{g(t)}-f|^{n^2\over 2(n-2)} dV_{g(t)}\right) \le Ce^{t}. $$
This proves Proposition 2.2.

\end{proof}

\bigskip

With the estimates of Proposition 2.1 one would like to apply the classical Shauder estimates for parabolic equations. To this end  we need $C^{\alpha}$-estimates :

\medskip

\begin{prop} Let $g(t) =  u(t)^{4\over n-2}g_0$ be the solution of $(1.3)$ defined on a maximal interval $[0, T^*)$. Then we have
for some $\alpha \in (0, 1)$  and any $T \in [0, T^*)$
$$\| u\|_{C^{\alpha}([0, T] \times M)} \le Ce^{CT} $$

\noindent where  $C$ is a positive constant depending only on $u_0, g_0$ and $f$.

\end{prop}

\begin{proof} By using Proposition 2.1 and Proposition 2.2, the proof is identical to that of Proposition 2.6 in Brendle \cite{sB1}.

\end{proof}

\begin{proof}[Proof of Theorem 1.1]  Let $g(t) =  u(t)^{4\over n-2}g_0$ be the solution of $(1.3)$ defined on a maximal interval $[0, T^*)$.  Assume by contradiction that $T^* < + \infty$. Then by  using Proposition 2.1 and Proposition 2.3  we have

$$ \|u\|_{C^{\alpha}([0, T^*)\times M)}  \le Ce^{CT^*}  \  \   \hbox{and} \  \  \min_{[0, T^*)\times M}u \ge \min(C_0, \min_M u_0). $$

\noindent for some $\alpha \in (0, 1)$, where $C$ is a positive constant depending $u_0, f, g_0$.  The classical theory of linear parabolic equations applied to (1.4) implies that $u$ is bounded in $C^{k}([0, T^*)\times M)$ for any $k \in \N$, that's
$$\|u\|_{C^{k}([0, T^*)\times M)} \le C_{k} \ , \eqno (2.16)$$
where $C_k$ is a positive constant depending only on $u_0, g_0, f$ and $k$. It is clear that (2.16) allows us to extend the solution beyond $T^*$ contradicting thus the maximality of $T^*$.  We see from (2.2) that the functional ${\mathcal E}$ is decreasing along the flow. The proof of Theorem 1.1 is then complete.
\end{proof}

\bigskip

\section{Long Time behavior of the flow}

\bigskip

In this section we study the asymptotic behavior of the flow $g(t)$ when $t \to +\infty$. First we prove the following proposition which gives a super solution of equation (1.1) when conditions (H1)-(H2) are satisfied.

\bigskip

\begin{prop}
Suppose that there exists an open set $\Omega \subset M$ such that conditions $(H1)-(H2)$ are satisfied. Then there exists a conformal metric ${\bar g} = {\bar u}^{4\over n-2}g_0$,  \ $0< {\bar u} \in C^{\infty}(M)$,  satisfying

$$ R_{\bar g} - f \ge 0 \eqno (3.1)$$
or equivalently
$$-c_n\Delta {\bar u} +  R_0 {\bar u}  -  f{\bar u}^{N} \ge 0 , \  N= {n+2\over n-2}. \eqno (3.2)$$
 \end{prop}

 \bigskip

 \begin{proof} By hypothesis, there is an open set $\Omega \subset M$ satisfying $(H1)-(H2)$, that's

 $$     \lambda_{\Omega} >  0   \  \    \hbox{and} \  \   f < 0 \  \hbox{on} \  M\setminus \Omega   \leqno  \bf{(H1) } $$

 \noindent and

$$   \sup_{x\in \Omega} f(x) \le C_{\Omega}\inf_{x\in M\setminus\Omega}\left|f(x)\right|,  \leqno \bf{(H2)}$$
\medskip

\noindent where $C_{\Omega}$ is a positive constant depending only on $\Omega$.

\medskip

\medskip

Let $\varepsilon >0$ and set
$$\Omega_{\varepsilon} = \{ x \in M \ : \  d(x, \Omega) < \varepsilon \ \}.$$

\medskip

For $\varepsilon >0$ sufficiently small we have from (H1) that $\lambda_{\Omega_{\varepsilon}} > 0$, where $\lambda_{\Omega_{\varepsilon}}$ is the first eigenvalue of the operator $-c_n\Delta + R_0$ on $\Omega_{\varepsilon}$  with zero Dirichlet boundary conditions. Let $D \subset M$  be an open set of smooth boundary such  that $\overline{\Omega} \subset  D  \subset  \Omega_{\varepsilon}$. Then we have $\lambda_{D} \ge \lambda_{\Omega_{\varepsilon}} > 0$.  Let $\varphi_{0}$ an eigenfunction  associated  with $\lambda_{D} $, that's

$$-c_n\Delta \varphi_0 + R_0\varphi_0 = \lambda_{D}\varphi_0 . $$
\medskip

Then we have that $\varphi_0 \in C^{\infty}(\overline{D})$ and using the maximum principle of elliptic equations one has $\varphi_0 >0 $ on   $D$.
By normalising if necessary, we may suppose that

$$0< \varphi_0 \le 1 \  \  \hbox{on} \   D.  \eqno (3.3) $$

\medskip

Let $\chi \in C^{\infty}_0(D)$ such that $ 0 \le \chi \le 1$ and $\chi = 1$ on $\overline{\Omega}$.  We define the function ${\bar u} \in C^{\infty}(M)$ by setting

$${\bar u} = \delta \left( \chi \varphi_0 + 1-\chi\right), $$

\medskip

\noindent where $\delta >0$ will be chosen later.  By (3.3) and the definition of $\chi$ it is easy to check that

$$m_0 := \inf_M \left( \chi \varphi_0 + 1-\chi\right) > 0, $$
so
$$ {\bar u} \ge \delta \hspace{0.3mm} m_0.  \eqno (3.4) $$

\bigskip

Now let us prove that ${\bar u}$ satisfies (3.2). If we set

$${\mathcal L}({\bar u}) = -c_n\Delta {\bar u} +  R_0 {\bar u}  - f{\bar u}^{n+2 \over n-2}, $$

\noindent then (3.2) is equivalent to ${\mathcal L}({\bar u}) \ge 0$.

\medskip

A simple computation shows that we have on $\Omega$ (using the fact that $\chi = 1$ on $\overline{\Omega}$ ) :

$${\mathcal L}({\bar u})  = \lambda_{D}\delta  \varphi_0 - f \delta^{N}\varphi_0^N =  \delta\varphi_0\left( \lambda_{D}- \delta^{N-1}f\varphi_0^{N-1} \right) $$

\medskip

\noindent and by using (3.3) it follows that

$${\mathcal L}({\bar u})  \ge \delta\varphi_0\left( \lambda_{D}- \delta^{N-1}\sup_{x\in \Omega}f(x) \right).   \eqno (3.5) $$

\medskip


It follows from (3.5) that if we want ${\mathcal L}({\bar u})  \ge 0$  on $\Omega$, we have to choose $\delta >0$ satisfying

$$\delta^{N-1}\sup_{x\in \Omega}f(x) \le  \lambda_{D} . \eqno (3.6)$$

\medskip

\medskip

Now we examine the sign of  ${\mathcal L}({\bar u})$  on $M\setminus \Omega$.  We have   from the definition of ${\bar u}$ that

$${\mathcal L}({\bar u}) =\delta\left( -c_n \Delta + R_0 \right)(\chi \varphi_0 + 1-\chi )- f {\bar u}^N . \eqno (3.7)$$

\bigskip

\noindent By using (3.4) and the fact that $f < 0$ on $M \setminus \Omega$, it follows from (3.7)

\bigskip

$${\mathcal L}({\bar u}) \ge-  \delta \hspace{0.3mm} m_1 +  \delta^{N} m_0^{N} \inf_{x \in M \setminus \Omega}|f(x)|   $$
where

$$ m_1 = \sup_M\left| \left( -c_n \Delta + R_0 \right)(\chi \varphi_0 + 1-\chi )\right| .$$

\medskip

Thus, if we want ${\mathcal L}({\bar u}) \ge 0$ on $M\setminus \Omega$,  we have to assume that

$$ -m_1 +  \delta^{N-1} m_0^{N} \inf_{x \in M \setminus \Omega}|f(x)|  \ge 0$$
that's
$$\delta^{N-1}\inf_{x \in M \setminus \Omega}|f(x)| \ge  m_1 m_0^{-N}. \eqno (3.8)$$

It is clear that the existence of $\delta >0$ satisfying both (3.6) and (3.8)  is  equivalent to condition (H2) with $C_{\Omega} = {\lambda_D m_0^N \over m_1}$.  This achieves the proof of Proposition 3.1.

\end{proof}

\bigskip

Proposition 3.1 allows us to prove uniform $L^{\infty}$-estimates on the flow.

\bigskip

\begin{prop} Let $0< u_0 \in C^{\infty}(M)$ such that $u_0 \le {\bar u}$ where ${\bar u}$ is given by Proposition 3.1. Then  the solution  $u$ of (1.4) satisfies for any $(t,x)  \in [0, +\infty)\times M$

$$ \min\left(C_0, \min_M u_0\right)  \le u(t,x) \le \max_{M}{\bar u}, \eqno (3.9) $$
where $\displaystyle C_0 = \left({\min_{M} |R_0| \over \max_{M}|f| }\right)^{n-2\over 4}$.
\end{prop}

\bigskip

\begin{proof}
First observe that the first inequality in (3.9) is already proved in Proposition 2.1. It remains then to prove the second inequality, that's,
$$ u(t,x) \le \max_{M}{\bar u}. $$

Let $v = {\bar u} - u$. Since $u$ satisfies (1.4) and ${\bar u}$ satisfies (3.2), then we have

$$\partial_t \left({\bar u}^N - u^N\right)  \ge {n+2 \over 4}\left(c_n\Delta v -R_0 v +  f\left({\bar u}^N- u^N\right)\right).  \eqno (3.10)$$

We have \   ${\bar u}^N - u^N = av$, where

$$a(t,x) = N\int_0^1\bigl(s \hspace{0.2mm}{\bar u}(t,x) + (1-s)u(t,x)\bigr)^{N-1} ds,$$

so it follows from (3.10) that

$$\partial_t(av)  \ge {n+2 \over 4}\left(c_n\Delta v -R_0 v + afv\right).   \eqno (3.11) $$

\medskip

Since $v(0, x) = {\bar u}(x) -u_0(x) \ge 0$, then by applying the maximum principle to (3.11)  we get  $v(t,x) \ge 0$ for any $t \ge 0$, that's
$$u(t, x) \le {\bar u}(x) .$$

Proposition 3.2 is then proved.
\end{proof}

\bigskip

Now we prove that the integral estimate (2.6) in Proposition 2.2 can be improved when $t \to +\infty$. More precisely, we have

\medskip

\begin{prop}
 Let $0< u_0 \in C^{\infty}(M)$ such that $u_0 \le {\bar u}$ where ${\bar u}$ is given by Proposition 3.1. Let $g(t)$ the solution of  (1.3) given by Theorem 1.1 such that $g(0) =u_0^{4\over n-2}g_0$.   Then we have for any $p \ge 1$
$$\lim_{t\to +\infty}\int_M |R_{g(t)}-f|^{p} dV_{g(t)}  = 0. \eqno (3.12) $$
\end{prop}

\medskip

\begin{proof}
In what follows $C$ denotes a positive constant depending only on $u_0, g_0, f , p $, and its value  may change from line to line.

\bigskip

We have by (2.2) for any $t \ge 0$,

$${n-2 \over 2}\int_0^t \int_M|R_g-f|^2 dV_g = {\mathcal E}(g(0)) - {\mathcal E}(g(t)). \eqno (3.13)$$

On the other hand, we have

$${\mathcal E}(g(t))  = \int_{M} \left(c_n\vert \nabla u\vert^2 + R_0u^2  -{n-2\over n}  f u^{2n\over n-2}\right) dV_{g_0},$$
and since    $u$ is uniformly bounded by Proposition 3.2, then we have $ {\mathcal E}(g(t))  \ge - C$. So it follows from (3.13) that

$$ \int_0^{+\infty}\hspace{-2mm}\int_M|R_{g(t)}-f|^2 dV_{g(t)} \le C.  \eqno (3.14)$$

Since by Proposition 3.2 the volume of $g(t)$ is  uniformly bounded, then it suffices to prove  (3.12) for a sequence $p_k \to +\infty$. We shall prove (3.12) by induction  when  $p = p_k$, where
$$ p_k := {n\over 2}\left({n\over n-2}\right)^k, \ k \in  \N . $$

First we prove (3.12) for $p_0 = {n\over 2}$. As  in the proof of Proposition 2.2, if we use Lemma 2.1  and the fact that $u$ is uniformly bounded by Proposition 3.2, then one has  for any $p >1$ :

$${d \over dt}\int_M |R_g-f|^{p} dV_g  + C^{-1} \left(\int_{M} |R_g - f|^{p n\over n-2} dV_{g}\right)^{n-2 \over n}  \le C \int_M |R_g-f|^{p} dV_g  $$
$$+  \      \left(p-{n\over 2} \right)\int_M |R_g-f|^{p+1} dV_g  .  \eqno (3.15) $$

Set
$$\phi_p(t) = \int_M |R_g-f|^{p} dV_g .$$

If $p_0 <2$, then by using H\"older's inequality and the fact that $u$ is uniformly bounded, we have
$$\phi_{p_0} \le C \phi_2^{p_0\over 2}.   \eqno (3.16)$$
 So it follows  from (3.15) by taking $p= p_0 = {n\over 2}$ that
$${d \over dt} \phi_{p_0}^{2\over p_0} \le C \phi_2 .  \eqno (3.17)$$

By (3.14) there is a sequence $t_{\nu}\to + \infty$ such that $\phi_2(t_{\nu}) \to 0$ and $\int_{t_{\nu}}^{+\infty}\phi_2(s) ds \to 0$. So by integrating (3.17)  on $[t_{\nu}, t]$ and using (3.16)  we get

$$\phi_{p_0}^{2\over p_0}(t)  \le \phi_{p_0}^{2\over p_0}(t_{\nu}) + C \int_{t_{\nu}}^{t} \phi_2(s) ds  \le  C  \phi_2(t_{\nu}) +  C \int_{t_{\nu}}^{t} \phi_2(s) ds  $$
Letting $t \to +\infty$ and $\nu \to + \infty$ we obtain $\phi_{p_0}(t)  \to 0$ as $t\to + \infty$.

\bigskip

If $p_0 \ge 2$, by using H\"older's inequality and  Young's inequality we have for any $\varepsilon >0$,
$$\int_M |R_g-f|^{p_0} dV_g  \le \varepsilon \left(\int_{M} |R_g - f|^{p_0 n\over n-2} dV_{g}\right)^{n-2 \over n} + \varepsilon^{-{n(p_0-2)\over4}} \left(\int_M |R_g-f|^{2} dV_g\right)^{p_0\over 2}. $$

By taking  \ $\varepsilon = {1\over 2}C^{-1}$, where $C^{-1}$ is the constant appearing in (3.15), we obtain from (3.15)(where we take $p= p_0 = {n\over 2}$),

$${d \over dt} \phi_{p_0} + C^{-1}\phi_{p_0n\over n-2}^{n-2\over n}  \le C \phi_{2}^{p_0\over 2}. \eqno (3.18)$$
But by H\"older's inequality, since the volume of $g$ is uniformly bounded, we have $\phi_2 \le C \phi_{p_0}^{2\over p_0}$  and $\phi_{p_0} \le C \phi_{p_0n\over n-2}^{n-2\over n} $.  Thus it follows from (3.18) that
$${d \over dt} \phi_{p_0}^{2\over p_0}+ C^{-1} \phi_{p_0}^{2\over p_0} \le C \phi_2  .  \eqno (3.19)  $$
If we integrate (3.19) on $[0, t]$ and using (3.14) we get
$$ \int_0^t \phi_{p_0}^{2\over p_0}(s) ds  \le  C, $$

\noindent which implies, since $t\ge 0$ is arbitrary,

$$ \int_0^{+\infty} \phi_{p_0}^{2\over p_0}(s) ds  \le  C.$$

Thus there exists a sequence $t_{\nu} \to + \infty$ such that $\phi_{p_0}^{2\over p_0}(t_{\nu}) \to 0 $ as $\nu \to + \infty$. If  we intgerate again (3.19) on $[t_{\nu} , t]$,  we obtain
$$\phi_{p_0}^{2\over p_0}(t) \le \phi_{p_0}^{2\over p_0}(t_{\nu}) + C \int_{t_{\nu}}^{t} \phi_2(s) ds . $$

 \noindent By using (3.14), it follow that $\phi_{p_0}^{2\over p_0}(t) \to 0$ as $t \to +\infty$.

\medskip

\medskip

\medskip

Now suppose by induction that

$$ \lim_{t\to +\infty}\phi_{p_k}(t)  = 0. \eqno (3.20)$$

First let us prove that

$$\lim_{t\to +\infty} \int_t^{t+1}\phi_{p_{k+1}}^{n-2\over n}(s) ds = 0 . \eqno (3.21)$$

We may suppose $k \ge 1$. Indeed,  if $k=0$ (that's $p_k = p_0 =  {n\over 2}$),  then (3.21) follows directly from  (3.15)(with $p= {n\over 2}$) by integrating on $[t, t+1]$ and using (3.20).  Thus let us prove (3.21) when $k \ge 1$.

\medskip

By using H\"older's inequality and Young's inequality we have for any $p >{n\over 2}$ and $\varepsilon >0$,
$$\int_M |R_g-f|^{p+1} dV_g  \le \varepsilon \left(\int_{M} |R_g - f|^{p n\over n-2} dV_{g}\right)^{n-2 \over n} + \varepsilon^{-{n\over 2p-n}} \left(\int_M |R_g-f|^{p} dV_g\right)^{1 + {2  \over 2p-n }} \eqno (3.22)$$

By taking $p =p_k$, \ $\varepsilon = {1\over 2}C^{-1}$, where $C$ is the constant appearing in (3.15), we obtain from (3.15)
$$ {d \over dt} \phi_{p_k}  + {1\over 2}C^{-1} \phi_{p_{k+1}}^{n-2\over n} \le C \phi_{p_k}^{1 + { 2 \over 2p_k-n}} +  C \phi_{p_k} $$
Then (3.21) follows by integrating on $[t, t+1]$ and using (3.20).

\medskip

Now if  we apply  (3.22) by taking $p = p_{k+1}$ and $\varepsilon = {C^{-1} \over p_{k+1}-{n\over 2}}$, where $C$ is the constant appearing in (3.15), we obtain from (3.15) (where we take $p= p_{k+1}$)
$$ {d \over dt} \phi_{p_{k+1}}   \le C \phi_{p_{k+1}}^{1 + \alpha_k} + C \phi_{p_{k+1}},  $$
where $\alpha_k = {2 \over 2p_{k+1}-n}$. The last inequality  is equivalent  to
$$ {d \over dt}\log \phi_{p_{k+1}}  \le   C\left( \phi_{p_{k+1}}^{\alpha_k}  + 1 \right).   \eqno (3.23)  $$

By (3.21) there is a sequence $t_{\nu} \to + \infty$ such that  $\nu \le  t_{\nu} \le \nu + 1$  satisfying $\phi_{p_{k+1}}(t_{\nu}) \to 0$ as $\nu \to +\infty$.  If we integrate (3.23) on $[t_{\nu}, t]$ where $t \in [ \nu , \nu+ 1] $, we obtain

$$\log{\phi_{p_{k+1}}(t)\over\phi_{p_{k+1}}(t_{\nu})}  \le  C \left( \int_{\nu}^{\nu+1} \phi_{p_{k+1}}^{\alpha_k}(s) ds + 1 \right) . \eqno (3.24)  $$

We note here that $ \alpha_k \le  {n-2 \over n}$, so by H\"older's inequality we have

$$  \int_{\nu}^{\nu+1} \phi_{p_{k+1}}^{\alpha_k}(s) ds \le \left(\int_{\nu}^{\nu+1} \phi_{p_{k+1}}^{n-2\over n}(s) ds  \right)^{{n\alpha_k \over n-2}} \to 0  \  \hbox{as}  \  \nu \to +\infty$$
by (3.21). Thus it follows from (3.24)
$$\log{\phi_{p_{k+1}}(t)\over\phi_{p_{k+1}}(t_{\nu})}  \le  C$$

\noindent which implies that  $\phi_{p_{k+1}}(t) \to 0 $ as $t \to + \infty$.   The proof of Proposition 3.3 is then complete.

\end{proof}

\bigskip

Now we  can prove uniform $C^{\alpha}$-estimates on the solution.

\bigskip

\begin{prop}
Let $0< u_0 \in C^{\infty}(M)$ such that $u_0 \le {\bar u}$ where ${\bar u}$ is given by Proposition 3.1. Then  the solution  $u$ of (1.4) satisfies
for some $\alpha \in (0, 1)$
$$\| u\|_{C^{\alpha}([0, +\infty) \times M)} \le C, $$

\noindent where  $C$ is a positive constant depending only on $u_0, g_0$ and $f$.

\end{prop}

\begin{proof}
By using Proposition 3.2  and Proposition 3.3, the proof is identical to that of Proposition 2.6 in Brendle \cite{sB1}.
\end{proof}

\bigskip

Now we are in position to prove Theorem 1.2.

\medskip

\begin{proof}[Proof of Theorem 1.2]

Let $g = u^{4\over n-2} g_0$ the solution of (1.3) given by Theorem 1.1. By Proposition 3.2 we have that $u$ is bounded from below and above  uniformly  on $[0, +\infty)$.   As in the proof of Theorem 1.1, this implies  that equation (1.4) is uniformly parabolic and  by Proposition 3.4   we have uniform $C^{\alpha}$-bound on the solution $u$ on $[0, +\infty)\times M$. We then apply the classical regularity theory of linear parabolic equations to obtain uniform $C^{k}$-bound for any $k \in \N$, that's
$$\|u(t)\|_{C^k(M)} \le C_k , \eqno (3.25) $$

\noindent for some constant $C_k$ independent of $t$. It follows from (3.25)  that there is  a sequence  $t_{\nu} \to + \infty$  such that $u(t_{\nu})$ converges in $C^k(M)$ for any $k \in \N$, to some function $ u_{\infty} \in
C^{\infty}(M)$. Since $u(t)$ is uniformly bounded from below by Proposition 3.2, then we have $u_{\infty} >0$.  By using Proposition 3.3 and Passing to the limit when $\nu \to \infty$,  we see that $ R_{g_{\infty}} = f$, where $g_{\infty} = u_{\infty}^{4\over n-2} g_0$, that's $f$ is the scalar curvature of $g_{\infty}$.  By the general result of Simon \cite {lS} on evolution equations, $u_{\infty}$ is the unique limit of $u(t)$ when $ t \to + \infty$.

\end{proof}

\bigskip

\begin{proof}[Proof of Corollary 1.1]

Since $f <0$ almost everywhere on $M$, then for $\varepsilon >0$ small enough, the open set
$$\Omega_{\varepsilon} = \{ \ x \in M \ : \  f(x) > -\varepsilon \ \}$$
 has arbitrary small volume.  This implies that the first eigenvalue $\mu_{\Omega_{\varepsilon}}$ of $-c_n\Delta$   on $\Omega_{\varepsilon}$ with zero Dirichlet conditions is arbitrary  large if $\varepsilon$ is small enough.  But since
 $$\lambda_{\Omega_{\varepsilon}}  \ge \mu_{\Omega_{\varepsilon}} + \min_{M}R_0 $$
 then we have $ \lambda_{\Omega_{\varepsilon}} >0$ if $\varepsilon $ is small enough. Thus the condition (H1) is satisfied with $\Omega = \Omega_{\varepsilon}$. Condition (H2) is also satisfied since by continuity of $f$ we have $f \le 0$ everywhere on $M$.

  \end{proof}

\begin{proof}[Proof of Theorem 1.3]
Suppose that condition (H1) is not satisfied, that's, for any open set $\Omega \subset M$ such that $f <  0$ on $M\setminus\Omega$,  we suppose  $\lambda_{\Omega} \le 0$.
For $\varepsilon >0$, consider the following familly of open sets   :
$$\Omega_{\varepsilon} = \{ \ x \in M \ : \  f(x) > -\varepsilon \ \} . $$

For the simplicity of notation we set  $\lambda_{\varepsilon} =  \lambda_{\Omega_{\varepsilon}}$. According to our hypothesis we have then

$$\lambda_{\varepsilon} \le 0 \   \  \hbox{for all} \  \varepsilon > 0.  \eqno (3.26)$$

By using Sard's theorem, there exists a sequence $\varepsilon_n \to 0$ such that $\varepsilon_n$    is a regular value of $f$ and then $\Omega_{\varepsilon_n}$ has  a smooth boundary $\partial\Omega_{\varepsilon_n}  = \{ x \in M \ : \  f(x) = -\varepsilon_n \ \}$.

\medskip

Let $\varphi_n$ an eigenfunction of $-c_n\Delta + R_0$ associated with $\lambda_{\varepsilon_n}$. As already mentioned in the proof of proposition 3.2, we have by the maximum principle that
$$\varphi_n > 0 \  \hbox{on} \  \Omega_{\varepsilon_n} \  \hbox{and} \   {\partial \varphi_n \over \partial \nu } \le 0   \
\hbox{on} \  \partial \Omega_{\varepsilon_n} \eqno (3.27) $$
where $\nu$ is the outer normal vector to $\partial\Omega_{\varepsilon_n}$.  By normalising if necessary, we may assume that
$$\int_{\Omega_{\varepsilon_n}}\varphi_n \hspace{1mm} dV_{g_0} = 1  . \eqno (3.28)  $$

If we multiply equation (1.4) by $\varphi_n$ and integrate on $\Omega_{\varepsilon_n}$, we have
$${d\over dt}\int_{\Omega_{\varepsilon_n}}u^N\varphi_n \hspace{0.5mm} dV_{g_0} = {n+2\over 4} \int_{\Omega_{\varepsilon_n}}(c_n\Delta u - R_0 u)\varphi_n  \hspace{0,5mm} dV_{g_0} +
{n+2\over 4} \int_{\Omega_{\varepsilon_n}} fu^N \varphi_n  \hspace{0,5mm} dV_{g_0} . \eqno (3.29) $$

An integration by parts gives

$$\int_{\Omega_{\varepsilon_n}}(c_n\Delta u - R_0 u)\varphi_n  \hspace{0,5mm} dV_{g_0} = -\lambda_{\varepsilon_n}\int_{\Omega_{\varepsilon_n}} u\varphi_n \hspace{0,5mm} dV_{g_0}  - c_n\int_{\partial \Omega_{\varepsilon_n}}{\partial \varphi_n \over \partial \nu } u \hspace{0,5mm} dV_{g_0} .$$

\medskip

Since $\lambda_{\varepsilon_n} \le 0$, then we obtain by using (3.27)

$$\int_{\Omega_{\varepsilon_n}}(c_n\Delta u - R_0 u)\varphi_n  \hspace{0,5mm} dV_{g_0}  \ge  -  \lambda_{\varepsilon_n} \inf_{M}u  -  c_n \inf_{M}u \int_{\partial \Omega_{\varepsilon_n}}{\partial \varphi_n \over \partial \nu } \hspace{0,5mm} dV_{g_0}  \eqno (3.30)$$

 On the other hand we have

 $$c_n\int_{\partial \Omega_{\varepsilon_n}}{\partial \varphi_n \over \partial \nu }  \hspace{0,5mm} dV_{g_0} = c_n\int_{\Omega_{\varepsilon_n}}\Delta\varphi_n  \hspace{0,5mm} dV_{g_0} = \int_{\Omega_{\varepsilon_n}}(-\lambda_{\varepsilon_n} + R_0 ) \varphi_n \hspace{0,5mm} dV_{g_0} $$
 and by using (3.28) we get
$$ - c_n\int_{\partial \Omega_{\varepsilon_n}}{\partial \varphi_n \over \partial \nu }  \hspace{0,5mm} dV_{g_0} \ge  \lambda_{\varepsilon_n} + \inf_M R_0 . \eqno (3.31) $$

 Combining (3.30) and (3.31) we obtain

  $$\int_{\Omega_{\varepsilon_n}}(c_n\Delta u - R_0 u)\varphi_n  \hspace{0,5mm} dV_{g_0}  \ge \inf_M R_0 \inf_Mu $$

 If we substitute in (3.29) we get
 $${d\over dt}\int_{\Omega_{\varepsilon_n}}u^N\varphi_n \hspace{0.5mm} dV_{g_0} \ge  {n+2\over 4} \inf_M R_0 \inf_Mu +
 {n+2\over 4} \int_{\Omega_{\varepsilon_n}} fu^N \varphi_n  \hspace{0,5mm} dV_{g_0} . \eqno (3.32) $$

By Proposition 3.2 we have $u \ge C_0$, where $C_0$ is a positive constant depending only on $u_0, g_0$ and $f$. It follows from  (3.32) by  using the fact that $f > -\varepsilon_n $ on $\Omega_{\varepsilon_n}$,

$${d\over dt}\int_{\Omega_{\varepsilon_n}}u^N\varphi_n \hspace{0.5mm} dV_{g_0} \ge    C - {n+2\over 4}\  \varepsilon_n \int_{\Omega_{\varepsilon_n}} u^N \varphi_n  \hspace{0,5mm} dV_{g_0} ,$$
where $C$ is a positive constant depending only on $u_0, g_0$ and $f$. By integrating this differential inequality on $[0, t]$, we get

$$ \int_{\Omega_{\varepsilon_n}}u^N(t)\varphi_n \hspace{0.5mm} dV_{g_0}  \ge
\int_{\Omega_{\varepsilon_n}}u_0^N\varphi_n \hspace{0.5mm} dV_{g_0} +   Ct -{ n+2\over 4}\  \varepsilon_n \int_0^t \int_{\Omega_{\varepsilon_n}} u^N(s) \varphi_n  \hspace{0,5mm} dV_{g_0} ds ,$$
which implies by using (3.28)
$$\max_{x\in M}u^N(t,x) \ge Ct - { n+2\over 4} \  \varepsilon_n \int_0^t  \max_{x\in M}u^N(s,x) ds.  $$
Letting $n \to +\infty$, we obtain
$$ \max_{x\in M}u^N(t,x) \ge Ct . $$

The proof of Theorem 1.3 is complete.

\end{proof}

\end{document}